\documentclass[a4paper]{amsart}

\usepackage{amsmath,amsfonts,amssymb,amsthm}
\usepackage{graphicx}
\usepackage{mathtools}
\usepackage[colorinlistoftodos]{todonotes}
\usepackage{cancel}
\usepackage{stmaryrd}
\usepackage{url}
\usepackage{ytableau}

\newtheorem{prop}{Proposition}
\newtheorem{theo}[prop]{Theorem}
\newtheorem{lem}[prop]{Lemma}
\newtheorem{cor}[prop]{Corollary}
\theoremstyle{definition}
\newtheorem{defin}[prop]{Definition}
\newtheorem*{remark}{Remark}

\def \be{\begin{equation*}}
\def \ee{\end{equation*}}
\def \tOmega{\widetilde\Omega}
\def \M{M_{z,z',\theta}}
\def \tM{\widetilde {M}_{z,z',\theta}}

\def \L{\mathfrak L}
\def \ve{\varepsilon}
\def \({\left(}
\def \){\right)}

\newcommand\restr[2]{{
  \left.\kern-\nulldelimiterspace 
  #1 
  \vphantom{\big|} 
  \right|_{#2} 
  }}

\title[Transition functions of processes on the Thoma simplex]{Transition functions of diffusion processes on the Thoma simplex}

\author[]{Sergei Korotkikh}

\date{}

\begin{document}

\maketitle

\begin{abstract}
The paper deals with a three-dimensional family of diffusion processes on an infinite-dimensional simplex. These processes were constructed by Borodin and Olshanski (2009; 2010), and they include, as limit objects, Ethier-Kurtz's infinitely-many-neutral-allels diffusion model (1981) and its extension found by Petrov (2009). 

Each process X from our family possesses a symmetrising measure M, called the z-measure. Our main result is that the transition function of X has continuous density with respect to M.  This is a generalization of earlier results due to Ethier (1992) and to Feng, Sun, Wang, and Xu (2011). Our proof substantially uses a special basis in the algebra of symmetric functions related to Laguerre polynomials. 
\end{abstract}

\begin{section}{Introduction}

Ethier and Kurtz \cite{EK} constructed a family of diffusion processes on the Kingman simplex

\be
\overline{\nabla}_\infty=\left\{\alpha=(\alpha_n)\ \bigg |\ \alpha_1\geq\alpha_2\geq\dots\geq0, \sum_i \alpha_i\leq1\right\},
\ee
depending on a parameter $\tau>0$ and preserving the Poisson-Dirichlet distributions $PD(\tau)$ \cite{Kin} (the Poisson-Dirichlet parameter is usually denoted by $\theta$ but we reserve this symbol for another purpose).  Ethier \cite{Et} showed that the  transition function of the Ethier-Kurtz diffusion has a  continuous density with respect to $PD(\tau)$.

Petrov \cite{Pe1} discovered a broader family of diffusions on $\overline{\nabla}_\infty$, which depend on two parameters $(a,\tau)$ and are related to the two-parameter Poisson-Dirichlet distributions $PD(a, \tau)$ \cite{Pit}. An extension of Ethier's theorem from \cite{Et} to Petrov's diffusions is contained in the work \cite{FSWX} by Feng, Sun, Wang, and Xu. 

The present paper deals with an even broader family of infinite-dimensional diffusion processes, constructed by Olshanski \cite{olshanski}, and earlier (in a special case) by Borodin and Olshanski in \cite{BoOl2}. These processes depend on a triplet of parameters $(z,z',\theta)$. Their state space is the Thoma simplex 

\be
\Omega=\left\{(\alpha,\beta)\in\mathbb R^\infty_{\geq0}\times\mathbb R^\infty_{\geq 0}\ \bigg|\ \alpha_1\geq\alpha_2\geq\dots; \beta_1\geq\beta_2\geq\dots; \sum_i\alpha_i+\sum_j\beta_j\leq1\right\}.
\ee

The (pre)generators of the processes are defined on the algebra $\mathbb R[p^\circ_2,p^\circ_3,\dots]$, where
\begin{equation}
\label{moment}
p^\circ_k=\sum_{i=1}^\infty\alpha_i^{k}+(-\theta)^{k-1}\sum_{j=1}^{\infty}\beta_j^k, 
\end{equation}
and are given by 

\begin{multline}
\label{Aexpressionzz}
A_{z,z',\theta}= \sum_{i,j\geq2}ij(p^\circ_{i+j-1}-p^\circ_ip^\circ_j)\frac{\partial^2}{\partial p^\circ_i\partial p^\circ_j}+\sum_{i\geq2}[(1-\theta)i(i-1)p^\circ_{i-1}\\+(z+z')ip^\circ_{i-1}-i(i-1)p^\circ_i-i\theta^{-1}zz'p^\circ_i]\frac{\partial}{\partial p^\circ_i}+\theta\sum_{i,j\geq1}(i+j+1)p^\circ_ip^\circ_j\frac{\partial}{\partial p^\circ_{i+j+1}}.
\end{multline}

In the limit as $\theta\to0$, the $\beta$-coordinates disappear from \eqref{moment}, so that the Thoma simplex turns into the Kingman simplex. Moreover, by letting  the parameters $(z,z')$ vary together with $\theta$ in an appropriate way one can degenerate the expression \eqref{Aexpressionzz} into the generator of Petrov's diffusion. In particular, in this way one can get the Ethier-Kurtz generator. 

Our main result is Theorem \ref{result}. It provides an extension of Ethier's theorem to the 3-parameter family of diffusions with generators \eqref{Aexpressionzz}.  We derive an explicit expression for the transition function, which shows, in particular that it has a continuous density with respect to the stationary distribution (the latter is called the z-measure with Jack's parameter).  As shown in a recent work by Olshanski \cite{support}, our result implies that the topological support of the z-measures is the whole Thoma simplex, leading to possible applications of independent interest (for example, see \cite{GoRo}).

Our approach uses ideas of \cite{Et} and \cite{FSWX} as well as some new arguments. The central role is played by the Laguerre symmetric functions with Jack's parameter $\theta$ which form bases in the algebra of symmetric functions.

The structure of this paper is as follows. In the beginning we recall several facts about Jack symmetric functions and the process with Jack parameter. In order to work with the z-measures, we introduce in section \ref{LaguerreS}  a basis of Laguerre symmetric functions. In section \ref{EigenExpansion} we describe the eigenfunction expansion of the transition density (Theorem \ref{result}), then we give some applications of our result.

I would like to express my gratitude to Grigori Olshanski for suggestions and remarks. This work has been funded by the Russian Academic Excellence Project '5-100'.
\end{section}

\begin{section}{Preliminaries on Jack's symmetric functions}
Here we will give a short reminder about Jack's symmetric functions, see \cite{macdonald} for further details. Let $\Lambda^{(N)}=\mathbb R[x_1, \dots, x_N]^{S_N}$ denote the graded algebra of symmetric polynomials in $x_1, \dots, x_N$. We define a projection map $\iota_N:\Lambda^{(N)}\to\Lambda^{(N-1)}$ which acts by setting $x_N=0$. The graded algebra of \emph{symmetric functions} $\Lambda$ is the projective limit of the graded algebras $\Lambda^{(N)}$ with the projection maps $\iota_N$. In other words, $\Lambda$ consists of symmetric formal power series in $x_1, x_2, \dots$ of finite degree. The graded degree $k$ component of $\Lambda$ is denoted by $\Lambda_k$. Note that we have a natural truncation map $\pi_N:\Lambda\to\Lambda^{(N)}$ setting $x_{N+1}=x_{N+2}=\dots=0$.

Let $\mathbb Y_n$ denote the set formed by partitions $\lambda$ of $n$, that is, the set of sequences of nonnegative integers $\lambda_1\geq\lambda_2\geq\dots$  s.t. $|\lambda|:=\sum_{i} \lambda_i=n$, and let $\mathbb Y=\bigsqcup_{n\geq 0}\mathbb Y_n$ denote the set of all partitions. Alternatively, we can view a partition $\lambda$ as a collection of boxes on a plane called \emph{Young diagram}, with row lengths $\lambda_i$. For a pair of partitions $\mu$ and $\lambda$ we write $\mu\subset \lambda$ if $\mu_i\leq \lambda_i$ for every $i$, and $\lambda/\mu$ denotes a skew diagram in the plane obtained by removing boxes of $\mu$ from the Young diagram $\lambda$ (see Figure \ref{youngdiagram}).

\begin{figure}
\begin{ytableau}
\ & & & & \\
 & & \\
 & 
\end{ytableau}
\hspace{1 cm}
\begin{ytableau}
\none&\none& & & \\
\none& & \\
 & 
\end{ytableau}
\caption{The Young diagram corresponding to $(5,3,2)$ (left) and the skew Young diagram corresponding to $(5,3,2)/(2,1)$ (right)}
\label{youngdiagram}
\end{figure}

The algebra $\Lambda$ has several distinguished bases enumerated by partitions. A basic example is the basis of monomial symmetric functions $m_\lambda$ given by
\be
m_\lambda:=\sum_{\alpha}\prod_{i=1}^{\infty}x_i^{\alpha_i},
\ee 
where the sum is over all distinct rearrangements $\alpha=(\alpha_1,\alpha_2,\dots)$ of the sequence $(\lambda_1, \lambda_2, \dots)$. Another basis $\{p_\lambda\}_{\lambda\in\mathbb Y}$ is generated by the power sums $p_k:=\sum_i x_i^k$ and is defined by
\be
p_\lambda:=\prod_{i=1}^{\infty} p_{\lambda_i}.
\ee
Here we set $p_0=1$. Note that $p_1, p_2, \dots$ are algebraically independent generators of $\Lambda$, hence $\Lambda\cong\mathbb R[p_1, p_2,\dots]$ (see \cite[I]{macdonald}).

In this work we will use a more complicated basis, or, more precisely, a family of bases $\{P_\lambda(x;\theta)\}_{\lambda\in\mathbb Y}$ depending on a positive real parameter $\theta$. The functions from this family are called Jack symmetric functions and they are uniquely characterised by the following properties:

1) $P_\lambda(x;\theta)=m_\lambda(x)+\sum_{\mu\leq\lambda}a_{\lambda\mu}m_\mu(x)$ where the sum is over partitions $\mu$ less than $\lambda$ in the lexicographical order.
 
2) $P_\lambda(x;\theta)$ are orthogonal with respect to the scalar product $\langle\cdot,\cdot \rangle_\theta$ on $\Lambda$ defined by
\be
\langle p_\lambda,p_\mu\rangle_\theta=\delta_{\lambda,\mu}z_\lambda\theta^{-l(\lambda)}.
\ee
Here we use notations from \cite{macdonald} where for a partition $\lambda$ we define $l(\lambda)$ as the number of nonzero parts in the partition $\lambda$, $m_i(\lambda)$ as the number of parts equal to $i$, and $z_\lambda$ is given by
\be
z_\lambda:=\prod_i i^{m_i(\lambda)} \cdot m_i(\lambda)!.
\ee

These properties uniquely define a system of homogenous symmetric functions in $\Lambda$. Moreover, for a fixed $\theta$ the functions $P_\lambda$ with $|\lambda|=n$ form a basis of $\Lambda_n$.

Let
\be
b^{(\theta^{-1})}_\lambda:=\langle P_\lambda(x;\theta),P_\lambda(x;\theta)\rangle_\theta^{-1},\quad Q_\lambda(x;\theta):=b_\mu^{(\theta^{-1})}P_\lambda(x;\theta),
\ee
so $\{Q_\lambda\}$ is the basis dual to $\{P_\lambda\}$, that is $\langle P_\lambda(x;\theta), Q_\mu(x;\theta)\rangle_{\theta^{-1}}=\delta_{\lambda,\mu}$.

The \emph{$\theta$-dimension} $\dim_\theta(\mu,\lambda)$ is defined as the coefficient in the expansion
\begin{equation}
\label{dimexp}
p_1^{n-|\mu|}P_{\mu}(x;\theta)=\sum_{|\lambda|=n}\dim_\theta(\mu,\lambda)P_{\lambda}(x;\theta).
\end{equation}
By duality this is equivalent to 
\begin{equation}
\label{dimtheta}
\dim_\theta(\mu,\lambda)=\langle p_1^{|\lambda|-|\mu|}P_\mu,Q_\lambda\rangle_\theta.
\end{equation}
When $\mu=\varnothing$ we simply write $\dim_\theta(\lambda)$. 
 
For further use we also recall several combinatorial notations. The \emph{Pochhammer symbol} is denoted by $(t)_n$ and is defined by
\be
(t)_n:=t(t+1)\dots(t+n-1).
\ee 

The $\theta$-content of a box $(i,j)$ is defined by
\be
c_\theta(i,j):=(j-1)-\theta(i-1).
\ee
We generalise the Pochhammer symbol defining it for a skew diagram $\lambda/\mu$ by 
\be
(z)_{\lambda/\mu,\theta}:=\prod_{(i,j)\in\lambda/\mu}(z+c_\theta(i,j)),
\ee
where the product is over boxes $(i,j)\in\lambda/\mu$ (that is, $\mu_i<j\leq \lambda_i$). For $\mu=\varnothing$ we will use the notation $(z)_{\lambda,\theta}:=(z)_{\lambda/\varnothing,\theta}$

\begin{remark}
The parameter $\theta$ is equal to $\alpha^{-1}$ in Macdonald's notation, see~\cite[VI]{macdonald} for further details on Jack symmetric functions.
\end{remark}
\end{section}

\begin{section}{Diffusion processes on Thoma simplex}
\label{thomaS}
\emph{The Thoma simplex} $\Omega$ is defined as the subspace of $\mathbb R^\infty_{\geq0}\times\mathbb R^\infty_{\geq0}$ formed by pairs $\omega=(\alpha,\beta)$ such that 
\be
\alpha_1\geq\alpha_2\geq\dots\geq0,\quad\beta_1\geq\beta_2\geq\dots\geq0,
\ee
\be
\sum_{i=1}^\infty\alpha_i+\sum_{j=1}^\infty\beta_j\leq 1,
\ee
with the topology induced from the product topology on $\mathbb R^\infty_{\geq0}\times\mathbb R^\infty_{\geq0}$.

For the rest of this section fix $\theta>0$. Let $C(\Omega)$ denote the space of continuous functions on $\Omega$ with supremum norm. We define an algebra morphism $\Lambda\to C(\Omega)$ by setting $p_1\mapsto 1$ and 
\be
p_k\mapsto \sum_{i=1}^\infty\alpha_i^{k} +(-\theta)^{k-1}\sum_{j=1}^\infty\beta_j^{k},\quad k\geq 2.
\ee 
The image of $f\in\Lambda$ in $C(\Omega)$ is denoted by $f^\circ$ and the resulting subalgebra in $C(\Omega)$ is denoted by $\Lambda^\circ$. The functions $p_n^\circ$ for $n\geq2$ are algebraically independent (see \cite[9.3]{olshanski}), hence $\Lambda^\circ=\mathbb C[p_2^\circ, \dots] \cong\Lambda/{(p_1-1)}$. Note that the grading $\Lambda=\bigoplus_{n\geq0}\Lambda_n$ induces a filtration on $\Lambda^\circ$:
\be
\Lambda^{\circ}_{\leq n}=\(\bigoplus_{i=1}^n \Lambda_i\)^\circ.
\ee 

We consider a family of Markov processes on $\Omega$ parametrized by triples $(z,z',\theta)$ such that one of the following conditions holds:

i) $z\in \mathbb C\backslash\mathbb R$ and $\overline z=z'$

ii) $\theta$ is rational and both $z$ and $z'$ are real numbers lying in one of the open intervals between two consecutive numbers from the lattice $\mathbb Z+\theta\mathbb Z$.

In the first (resp. second) case we say that $(z, z')$ belongs to the \emph{principal series} (resp. to the \emph{complementary series}). Later we will also need the \emph{degenerate series} defined as pairs $(z,z')=(N\theta, c+(N-1)\theta)$ for $N\in\mathbb Z_{>0}$ and $c>0$.

Define an operator $A_{z,z',\theta}$ on $\Lambda^\circ$ by 

\begin{equation*}
\begin{aligned}
A_{z,z',\theta}= &\sum_{i,j\geq2}ij(p^\circ_{i+j-1}-p^\circ_ip^\circ_j)\frac{\partial^2}{\partial p^\circ_i\partial p^\circ_j}+\sum_{i\geq2}[(1-\theta)i(i-1)p^\circ_{i-1}\\&+(z+z')ip^\circ_{i-1}-i(i-1)p_i^\circ-i\theta^{-1}zz'p_i^\circ]\frac{\partial}{\partial p_i^\circ}\\&+\theta\sum_{i,j\geq1}(i+j+1)p^\circ_ip^\circ_j\frac{\partial}{\partial p^\circ_{i+j+1}},
\end{aligned}
\end{equation*}
where $\frac{\partial}{\partial p_i^\circ}$ denotes a formal differentiation in $\Lambda^\circ\cong\mathbb R[p_2^\circ, p_3^\circ,\dots]$. This operator was introduced and studied in \cite{olshanski} and here we will recall its key properties. 
\begin{theo}[{\cite[Theorems 9.6, 9.10]{olshanski}}]
1) The operator $A_{z,z'\theta}$ is closable and its closure serves as a generator of a Feller Markov process $X_{z,z',\theta}$ on $\Omega$.

2) There exists a unique invariant probability measure for $X_{z,z',\theta}$.

3) Moreover, this measure is the symmetrizing measure. 
\end{theo}
The invariant measure mentioned in the theorem is denoted by $\M$ and is called \emph{z-measure} on the Thoma simplex. For any $\lambda$ It satisfies the following relation 
\begin{equation}
\label{zmeasurechar}
\int_\Omega  Q^\circ_\lambda(\omega;\theta)\M(d\omega)=\frac{\dim_\theta(\lambda)(z)_\lambda(z')_\lambda}{|\lambda|!(\theta^{-1}zz')_{|\lambda|}}.
\end{equation}
Moreover, the following theorem shows that these relations characterise $\M$.   

\begin{theo}[{\cite[Theorem B]{KOO}}]
The equation
\be
M(\lambda)=\dim_\theta(\lambda)\int_\Omega P_\lambda^\circ(\omega;\theta)\mathcal M(d\omega)
\ee
gives a one-to-one correspondence between probability measures $\mathcal M$ on $\Omega$ and nonnegative functions $M(\lambda)$ on $\mathbb Y$ such that $M(\varnothing)=1$ and 
\be
M(\lambda)=\sum_{\substack{\lambda\subset \mu\\ |\mu|=|\lambda|+1}}M(\mu)\frac{\dim_\theta(\lambda)\dim_\theta(\lambda,\mu)}{\dim_\theta(\mu)}.
\ee
\end{theo}
In particular, for every probability measure $\mathcal M$ on $\Omega$ we have
\be
\int_\Omega P^\circ_\lambda(\omega;\theta)\mathcal{M}(d\omega)\geq0
\ee
because $\dim_\theta(\lambda)>0$. Hence the functions $P^\circ_\lambda$ are nonnegative on $\Omega$. 
\end{section}

\begin{section}{Laguerre symmetric functions}
\label{LaguerreS}

In the previous section we have introduced measures $\M$, which play an important role in our work. In this section we will describe our main approach to working with these measures.

The classical Laguerre polynomials are eigenfunctions of the differential operator

\be
D^{(L)}=x\frac{d^2}{dx^2}+(c-x)\frac{d}{dx}.
\ee

These polynomials are orthogonal in $L^2(\mathbb R_{\geq0}, \gamma_c)$, where $\gamma_{c}$ is the gamma distribution defined by
\be
\gamma_{c}(dr)=\frac{1}{\Gamma(c)}r^{c-1}e^{-r}dr.
\ee  
There is an explicit formula for the Laguerre polynomials:
\be
L_n^c=\frac{c}{n!}\sum_{j=0}^n\binom{n}{j}\frac{(-x)^j}{(c)_j}.
\ee

Fix $\theta>0$. The \emph{generalised Laguerre polynomials} $L_\lambda^{c,N,\theta}(x_1,\dots, x_N)$ are symmetric polynomial eigenfunctions of $N$-variable generalization of $D^{(L)}$, given by
\be
D^{(L)}_N=\sum_{j=1}^N\left(x_j\frac{\partial^2}{\partial x_j^2} +(c-x_j)\frac{\partial}{\partial x_j} +2\theta\sum_{k\neq j}\frac{x_j}{x_j-x_k}\frac{\partial}{\partial x_j} \right).
\ee

They are described in \cite{BaFo} and we state several properties proved there. Recall that $P_\lambda$ and $Q_\lambda$ are the Jack symmetric functions, which can be viewed as symmetric polynomials in $N$ variables via the truncation map $\pi_N$.

\begin{prop}[{\cite[Proposition 4.3]{BaFo}}]
\label{Lagpolex}
We have
\be
L_\lambda^{c,N,\theta}(x)=\sum_{\mu\subseteq\lambda}(-1)^{|\lambda|-|\mu|}\frac{\dim_\theta(\mu,\lambda)}{(|\lambda|-|\mu|)!}(N\theta)_{\lambda/\mu,\theta}(c+(N-1)\theta)_{\lambda/\mu,\theta}Q_\mu(x;\theta),
\ee
\be
Q_\lambda(x;\theta)=\sum_{\mu\subseteq\lambda}\frac{\dim_\theta(\mu,\lambda)}{(|\lambda|-|\mu|)!}(N\theta)_{\lambda/\mu,\theta}(c+(N-1)\theta)_{\lambda/\mu,\theta}L_\mu^{c,N,\theta}(x),
\ee
where $x$ stands for $x_1,\dots, x_N$.
\end{prop}
Define a probability measure $\mu_{c,N}$ on the space $\mathbb R^N_{\mathrm{ord}} =\{x_1\geq\dots\geq x_N|x_i\in\mathbb R_{\geq 0}\}$ by
\begin{equation}
\label{dmu-def}
\mu_{c,N}(dx)=\mathrm{const\ }\prod_{i=1}^Nx_i^{c-1}e^{-x_i}\prod_{1\leq j\leq k\leq N}|x_j-x_k|^{2\theta}dx_1\dots dx_n.
\end{equation}
\begin{prop}[{\cite[Proposition 4.10]{BaFo}}]
\label{ortho}
The Laguerre symmetric polynomials are orthogonal in $L^2(\mathbb R^N_{\mathrm{ord}},\mu_{c,N})$. More precisely, we have
\be
\langle L_\lambda^{c,N,\theta},L_\mu^{c,N,\theta} \rangle=\delta_{\lambda,\mu}(N\theta)_{\lambda,\theta}(c+(N-1)\theta)_{\lambda,\theta}b_\lambda^{(\theta^{-1})}.
\ee
\end{prop}

\begin{remark}
Notation used in \cite{BaFo} is different from ours, see Appendix A for the match between the notations.
\end{remark}

Now using the same ideas as in \cite{DH},\cite{lag}, we will build "the analytic continuation" of the Laguerre symmetric polynomials taking $L_\lambda^{c,N,\theta}$ as a degeneration of some symmetric function with coefficients in $\mathbb C[z,z']$. Recall that $\pi_N$ denote the $N$-th truncation map $\Lambda\to\Lambda^{(N)}$.

\begin{theo}
\label{Lagexist}
Let $\lambda$ be a partition. There is a unique symmetric function $\L_\lambda$ in $\Lambda\otimes\mathbb C[z,z']$ such that 
\begin{equation}
\label{projection}
\pi_N\left(\restr{\L_\lambda}{z=N\theta,z'=c+(N-1)\theta}\right)=L^{c,N,\theta}_\lambda
\end{equation}
for any $N\geq l(\lambda)$ and $c>0$.
\end{theo}
\begin{defin}
The symmetric function $\L_\lambda$ is called the \emph{Laguerre symmetric function}.
\end{defin}
\begin{proof}[Proof of Theorem \ref{Lagexist}]
For every $\lambda$ such that $l(\lambda)\leq N$
\be
\pi_N(Q_\lambda(x;\theta))=Q_\lambda(x_1,\dots,x_N;\theta)\neq0.
\ee
Then the existence of the Laguerre functions follows from Proposition \ref{Lagpolex} because we can exhibit an expression satisfying \eqref{projection}:
\be
\L_\lambda(x)=\sum_{\mu\subseteq\lambda}(-1)^{|\lambda|-|\mu|}\frac{\dim_\theta(\mu,\lambda)}{(|\lambda|-|\mu|)!}(z)_{\lambda/\mu,\theta}(z')_{\lambda/\mu,\theta}Q_\mu(x;\theta).
\ee

To prove the uniqueness it is enough to show that the conditions
\begin{equation}
\label{conditions}
\pi_N\left(\restr{f}{z=N\theta,z'=c+(N-1)\theta}\right)=0
\end{equation}
force $f=0$. Assume that $f$ satisfies \eqref{conditions} and let $a_\lambda(z,z')=\langle f, Q_\lambda(x;\theta)\rangle_\theta\in\mathbb C[z,z']$ denote the coefficients in the expansion of $f$ in the basis of Jack functions. Then for every $N\geq l(\lambda)$ and $c>0$ we have
\be
a_\lambda(N\theta,c+(N-1)\theta)=0.
\ee
Then $a_\lambda(z,z')\equiv0$ because
\be
\{ (z,z')=(N\theta, c+(N-1)\theta) | N\geq l(\lambda),  c>0\}\subset \mathbb C^2
\ee
is a uniqueness set for polynomials in two variables. Hence $f=0$, which proves the uniqueness.
\end{proof}
\begin{cor}
We have
\begin{equation}
\label{lagdef}
\L_\lambda(x)=\sum_{\mu\subseteq\lambda}(-1)^{|\lambda|-|\mu|}\frac{\dim_\theta(\mu,\lambda)}{(|\lambda|-|\mu|)!}(z)_{\lambda/\mu,\theta}(z')_{\lambda/\mu,\theta}Q_\mu(x;\theta).
\end{equation}
\end{cor}

Note that $\L_\lambda(x)$ are  non-homogeneous symmetric functions with the top degree term equal to $Q_\lambda(x;\theta)$. To work with $\L_\lambda$ we will consider the natural filtration of $\Lambda\otimes \mathbb C[z,z']$ defined by
\be
\Lambda_{\leq n}\otimes \mathbb C[z,z']:=\bigoplus_{i=0}^{n}\Lambda_i\otimes\mathbb C[z,z'].
\ee

\begin{prop}
\label{PexpL}
The Laguerre functions $\L_\lambda$ with $|\lambda|\leq n$ form a $\mathbb C[z,z']$-basis of $\Lambda_{\leq n}\otimes \mathbb  C[z,z']$. In particular, 
\begin{equation}
\label{a}
P_\lambda(x;\theta)=\left(b_\mu^{(\theta^{-1})}\right)^{-1}\sum_{\mu\subseteq\lambda}\frac{\dim_\theta(\mu,\lambda)}{(|\lambda|-|\mu|)!}(z)_{\lambda/\mu,\theta}(z')_{\lambda/\mu,\theta}\L_\mu(x).
\end{equation}
\end{prop}
\begin{proof}
Note that \eqref{lagdef} and \eqref{a} imply that the transition matrix from $P_\lambda$ to $\L_\lambda$ is non-degenerate and upper-triangular (with respect to the order $\lambda\subseteq \mu$). Since the funcitons $P_\lambda(x;\theta)$ with $|\lambda|\leq n$ form a basis of $\Lambda_{\leq n}$, the functions $\L_\lambda(x)$ with $|\lambda|\leq n$ also form a basis of $\Lambda_{\leq n}\otimes \mathbb C[z,z']$. So it is enough to prove \eqref{a}. 

Using \eqref{lagdef} write the right hand side of \eqref{a} as a linear combination of $P_\mu$ with coefficients in $\mathbb C[z,z']$. Note that $\deg \L_\lambda=|\lambda|$, so only $P_\mu$ with $|\mu|\leq|\lambda|$ will occur. By Proposition \ref{Lagpolex} the equation holds for $(z,z')=(N\theta, c+(N-1)\theta)$ with $N\geq l(\lambda)$. Hence $\eqref{a}$ holds in $\Lambda[z,z']$ by the same argument as in the previous proof.
\end{proof}

In order to establish the orthogonality of the functions $\L_\lambda$ define the \emph{Thoma cone} $\tOmega$ as the subspace of $\mathbb R^\infty_{\geq0}\times\mathbb R^\infty_{\geq0}\times \mathbb R_{\geq0}$ consisting of triples $(\widetilde\alpha, \widetilde\beta, r)$ such that
\be
\widetilde\alpha_1\geq \widetilde\alpha_2\geq\dots\geq 0,\quad \widetilde\beta_1\geq\widetilde\beta_2\geq\dots\geq 0,
\ee
\be
\sum\widetilde\alpha_i+\sum\widetilde\beta_i\leq r.
\ee

An embedding of $\Lambda$ in the space of continuous functions on $\tOmega$ is defined by setting
\begin{gather*}
p_1\mapsto r,\\ 
p_k\mapsto \sum\widetilde\alpha_i^{k}+(-\theta)^{k-1}\sum\widetilde\beta_i^{k},\quad k\geq 2.
\end{gather*}

Note that the Thoma cone is a cone with the base $\Omega$ and the vertex $\tilde 0=(0,0,0)$, i.e. $\tOmega\backslash\{\tilde 0\}\cong\Omega\times\mathbb R_{>0}$. The isomorphism is given by sending a point $(\widetilde\alpha,\widetilde\beta,r)$ to $((\widetilde\alpha_i/r,\widetilde\beta_i/r),r)\in\Omega\times\mathbb R_{>0}$. The \emph{lifting} of the measure $\M$ is defined as a measure $\tM$ on $\tOmega\backslash\{\tilde 0\}=\mathbb R_{>0}\times\Omega$ equal to the product 
\be
\tM:=\M\otimes\gamma_{zz'\theta^{-1}},
\ee
where $\gamma_{zz'\theta^{-1}}$ is the Gamma distribution defined before.

It is readily seen that for any $f\in\Lambda_n$ we have
\begin{multline}
\label{decompPoc}
\int_{\tOmega} f(x)\tM(dx)=\int_{\mathbb R_{\geq0}}r^n\gamma_{\theta^{-1}zz'}(dr)\int_\Omega f^\circ(\omega)\M(d\omega)\\\frac{\Gamma(\theta^{-1}zz'+n)}{\Gamma(\theta^{-1}zz')}\int_\Omega f^\circ(\omega)\M(d\omega)=(\theta^{-1}zz')_n\int_\Omega f^\circ(\omega)\M(d\omega).
\end{multline}

The following proposition shows that $\tM$ is in fact an extrapolation of the measures $\mu_{c,N}$ defined in \eqref{dmu-def}.

\begin{prop}
For $(z,z')=(N\theta, c+(N-1)\theta)$ the measure $\tM$ degenerates to the measure $\mu_{c,N}$ on 
\be
\mathbb R^N_{\mathrm{ord}}=\{(\widetilde\alpha,\widetilde\beta,r)\in \tOmega : \alpha_{N+1}=\alpha_{N+2}=\dots=0, \beta=0, r=\alpha_1+\dots+\alpha_N\}.
\ee
\end{prop}
To prove this proposition consider measures $M_n(\lambda)$ defined on partitions of a number $n$ by 
\be
M_n(\lambda):=\frac{\dim_\theta(\lambda)^2(z)_{\lambda,\theta}(z')_{\lambda,\theta}}{|\lambda|!(\theta^{-1}zz')_{|\lambda|}b_\lambda^{(\theta^{-1})}}.
\ee
The measure $\M$ can be thought of as the limit of these measures and the proposition follows from a direct computation for $(z,z')=(N\theta, c+(N-1)\theta)$. See \cite[Remark 1.10]{BoOl1} and \cite[Section 12]{Ke1} for a more detailed explanation of this fact.

\begin{theo}
The functions $\L_\lambda$ are orthogonal in $L^2(\tOmega, \tM)$:
\be
\langle\L_{\lambda},\L_{\mu}\rangle_{L^2(\tOmega,\tM)}=\delta_{\lambda,\mu}(z)_{\lambda,\theta}(z')_{\lambda,\theta}b_\lambda^{(\theta^{-1})}.
\ee
\end{theo}
\begin{proof}
First note that 
 $\int_{\tOmega} P_\lambda(\omega) \tM(d\omega)$
is a rational function in $z, z'$ (by \eqref{zmeasurechar}). Hence for any $f\in\Lambda[z,z']$ 
\be
\int_{\tOmega} f(\omega) \tM(d\omega)\in\mathbb R(z,z').
\ee
In particular, $\langle\L_{\lambda},\L_{\mu}\rangle_{L^2(\tOmega,\tM)}$ is a rational function in $z,z'$. 

For $(z,z')=(N\theta, c+ (N-1)\theta)$ the orthogonality relations hold by Proposition \ref{ortho}. So these relations should hold for every $(z,z')$.
\end{proof}

\begin{cor}
\label{exprPL}
For any partitions $\lambda,\mu$
\be
\langle P_\lambda, \L_\mu\rangle_{L^2(\tOmega,\tM)}=\left(b_\lambda^{(\theta^{-1})}\right)^{-1}b_\mu^{(\theta^{-1})}\frac{\dim_\theta(\mu,\lambda)}{(|\lambda|-|\mu|)!}(z)_{\lambda,\theta}(z')_{\lambda,\theta}.
\ee
\end{cor}
\begin{proof}
Use the orthogonality relations for $\L_\lambda$ and Proposition \ref{PexpL}.
\end{proof}

\begin{remark}
The Laguerre symmetric functions for the case $\theta=1$ were studied in \cite{lag}. They were also described in \cite{DH} from a different perspective where the Laguerre symmetric functions were defined as eigenfunctions of the differential operator
\begin{multline*}
\mathcal D=\sum_{i=1}^\infty(-ip_i\frac{\partial}{\partial p_i} + (z+z')(i+1)p_i\frac{\partial}{\partial p_{i+1}}+(1-\theta)i(i+1)p_i\frac{\partial}{\partial p_{i+1}})+\frac{zz'}{\theta}\frac{\partial}{\partial p_1}\\+\sum_{i,j=1}^\infty(ijp_{i+j-1}\frac{\partial^2}{\partial p_i\partial p_j}+\theta(i+j+1)p_ip_j\frac{\partial}{\partial p_{i+j+1}}).
\end{multline*}
Similarly to the process on the Thoma simplex considered here, one can define a process on $\tOmega$ using $\mathcal D$ as a pregenerator, with $\tM$ being the symmetrizing measure of the process. See \cite{BoOl3} for more details on the resulting process on $\tOmega$.
\end{remark}

\end{section}

\begin{section}{Eigenfunction expansion of transition density}
\label{EigenExpansion}
In this section we follow the general scheme of proof in \cite{Et} to show the existence and the continuity of the transition density of $X_{z,z',\theta}$ with respect to $\M$.

Fix $(z,z',\theta)$ from either principal or complementary series and let $\langle\cdot,\cdot\rangle$ denote the scalar product in $L^2(\Omega,\M)$. First we will describe the spectral structure of $A_{z,z',\theta}$ as an essentially self-adjoint operator in $L^2(\Omega,\M)$.

\begin{theo}
\label{001}
1) The spectrum of the operator $A_{z,z',\theta}$ is purely discrete and is equal to $\{0, -\alpha_2, -\alpha_3, \dots\}$ where we set 
\be
\alpha_m:=m(m-1+zz'\theta^{-1}).
\ee
The multiplicity of $0$ is equal to $1$ and the multiplicity $d_m$ of $-\alpha_m$ is equal to the number of partitions of the number $m$ without parts equal to $1$.

2)We have the following decomposition into eigenspaces of $A_{z,z',\theta}$
\begin{equation*}
L^2(\Omega,\M)=\bigoplus^\infty_{m}W_m
\end{equation*}
where $m=0, 2, 3, \dots$ and $W_m$ is the eigenspace with the eigenvalue $-\alpha_m$.  Moreover, $\bigoplus^N_{m=0}W_m=\Lambda^\circ_{\leq N}$.
\end{theo}
The proof is due to G. Olshanski.
\begin{proof}
We will use the similar fact for operator $A_{z,z',\theta}$ acting on $\Lambda^\circ$. By {\cite[Theorem 9.9]{olshanski}} the action of operator $A_{z,z',\theta}$ on $\Lambda^\circ$ is diagonalizable with eigenvalues $\{0, -\alpha_2, -\alpha_3,\dots\}$ and multiplicities $d_m$ as in the theorem. 

Note that $\Lambda^\circ$ is well-defined as a subspace of $L^2(\Omega, \M)$. Indeed, we have a natural mapping from $\Lambda^\circ$ to $L^2(\Omega,\M)$. To show that it is injective, assume that $f^\circ=0$ in $L^2(\Omega,\M)$ for some $f\in\Lambda$. Multiplying homogeneous components of $f$ by powers of $p_1$, we may assume that $f$ is homogeneous of degree $n$. Then $f^\circ r^n=\widetilde {f}=0$, where $\widetilde{f}\in L^{2}(\tOmega,\tM)$ is the realisation of the symmetric function $f$ on the Thoma cone $\tilde\Omega$. But the map $f\mapsto\tilde f$ is injective, because $\{\L_\lambda\}$ is simultaneously a linear basis of $\Lambda$ and an orthogonal basis of $L^2(\tOmega,\tM)$. Hence $f=0$.

Recall that $\Lambda^\circ$ is a dense subspace of $C(\Omega)$. Then $\Lambda^\circ$ is a dense subspace of $L^2(\Omega,\M)$ and there is a decomposition of $\Lambda^\circ$ into eigenspaces of the symmetric operator $A_{z,z',\theta}$. Hence the decomposition of $\Lambda^\circ$ into eigenspaces of the symmetric operator $A_{z,z',\theta}$ extends to $L^2(\Omega,\M)$.
\end{proof}

So the pre-generator $A_{z,z',\theta}$  has an orthonormal eigenbasis $\{g_\lambda\}$, where $\lambda$ runs over all partitions without parts equal to $1$. Let $T(t)$ denote the semigroup on $C(\Omega)$ generated by $A_{z,z',\theta}$. Then for any $f\in \Lambda^\circ$ and $t>0$ we have
\be
T(t)f=\sum_{\lambda}e^{-\alpha_{|\lambda|}t}\langle f, g_\lambda\rangle g_\lambda.
\ee
Note that since $f\in \Lambda^\circ$ has a finite degree, only finite number of scalar products $\langle f, g_\lambda\rangle$ are nonzero, hence the sum above is finite. Writing the scalar product as an integral, we get
\be
T(t)f(\omega)=\sum_{\lambda}\int_{\Omega} e^{-\alpha_{|\lambda|}t}g_\lambda(\omega)g_\lambda(\sigma)f(\sigma)\M(d\sigma).
\ee
Hence if the series
\begin{equation}
\label{02}
p(t,\omega,\sigma)=\sum_{\lambda}e^{-\alpha_{|\lambda|}t}g_\lambda(\omega)g_\lambda(\sigma)
\end{equation}
absolutely converges then 
\begin{equation}
\label{003}
T(t)f(\omega)=\int_\Omega p(t,\omega,\sigma)f(\sigma)\M(d\sigma)
\end{equation}
for any $f\in\Lambda^\circ$ and $t>0$. By continuity, \eqref{003} will hold for any $f\in C(\Omega)$, so $p(t,\omega,\sigma)$ will be the transition density of the process $X_{z,z',\theta}$ with respect to $\M$. 

In order to prove the convergence of \eqref{02} we will express $p(t,\sigma,\omega)$ in terms of Jack's symmetric functions.  For $m\geq 1$ define $G_m\in C(\Omega\times\Omega)$ by 
\be
G_m(\omega, \sigma):=\sum_{|\lambda|=m}g_\lambda(\omega)g_\lambda(\sigma).
\ee
Note that $G_1=0$.

It turns out that the functions $G_m$ don't depend on the choice of eigenbasis and that they can be explicitly computed in terms of symmetric functions. Define $K_n^\circ\in C(\Omega\times\Omega)$ by
\be
K^\circ_n(\omega,\sigma):=\sum_{|\lambda|=n}b_\lambda^{(\theta^{-1})}\frac{P^\circ_\lambda(\omega;\theta)P^\circ_\lambda(\sigma;\theta)}{(z)_{\lambda,\theta}(z')_{\lambda,\theta}}.
\ee

Note that for $(z,z',\theta)$ from principal or complementary series the denominators in the definition of $K_n^\circ$ are positive real numbers.

\begin{lem}
\label{Lemma}
Let $n\geq m$ and $f\in\Lambda^\circ_{\leq m}$. Then
\be
\int_\Omega K^\circ_n(\omega,\cdot)f(\omega)\M(d\omega)-\frac{f(\cdot)}{(n-m)!(\frac{zz'}{\theta})_{m+n}}\in\Lambda^\circ_{\leq m-1}.
\ee
\end{lem}
This key lemma is similar to \cite[Lemma 3.2]{Et} and \cite[Lemma 3.2]{FSWX} but our proof is different and is based on the Laguerre symmetric functions introduced before.
\begin{proof}

First we will prove the equivalent property for the Thoma cone. For $n>0$ define functions $K_n(x,y)$ on $\tOmega\times\tOmega$ by
\be
K_n(x,y):=\sum_{|\lambda|=n}b_\lambda^{(\theta^{-1})}\frac{P_\lambda(x;\theta)P_\lambda(y;\theta)}{(z)_{\lambda,\theta}(z')_{\lambda,\theta}}.
\ee

Then using Corollary \ref{exprPL}
\begin{multline*}
\int_{\tOmega}K_n(x,y)\L_{\mu}(x)\tM(dx)=\sum_{|\lambda|=n}b_\lambda^{(\theta^{-1})}\frac{\langle P_{\lambda},\L_{\mu}\rangle_{L^2}}{(z)_{\lambda,\theta}(z')_{\lambda,\theta}}P_{\lambda}(y;\theta)\\
=b_\mu^{(\theta^{-1})}\sum_{|\lambda|=n}\frac{\dim_\theta(\mu,\lambda)}{(|\lambda|-|\mu|)!}P_{\lambda}(y;\theta)=b_\mu^{(\theta^{-1})}\frac{1}{(n-|\mu|)!}\sum_{|\lambda|=n}\dim_\theta(\mu,\lambda)P_{\lambda}(y;\theta).
\end{multline*}.

Hence by \eqref{dimexp}
\be
\int_{\tOmega}K_n(x,y)\L_{\mu}(x)\tM(dx)=b_\mu^{(\theta^{-1})}\frac{1}{(n-|\mu|)!}p_1^{n-|\mu|}(y)P_\mu(y;\theta).
\ee
Since the elements $\L_\mu$ with $|\mu|\leq m$ form a basis of $\Lambda_{\leq m}$, for every $f\in\Lambda_{\leq m}$ we have
\be
\int_{\tOmega}K_n(x,\cdot)f(x)\tM(dx)\in p_1^{n-m}\Lambda_m.
\ee
Finally, the $m$-th homogenous component of $\L_\mu$ is $b_\mu^{(\theta^{-1})}P_\mu$, hence
\be
\int_{\tOmega}K_n(x,\cdot)P_\mu(x;\theta)\tM(dx)-\frac{1}{(n-|\mu|)!}p_1^{n-m}P_\mu(\cdot;
\theta)\in p_1^{n-m+1}\Lambda_{m-1}.
\ee
Since the elements $P_\mu$ with $|\mu|=m$ form a basis of $\Lambda_m$, we get for any $f\in\Lambda_{\leq m}$
\be
\int_{\tOmega}K_n(x,y)f(x)\tM(dx)-\frac{1}{(n-m)!}p_1^{n-m}(y)f(y)\in p_1^{n-m+1}\Lambda_{m-1}.
\ee

To deduce the original lemma from the lifted analogue just proved, recall that every nonzero point $x\in\tOmega\backslash\{\tilde 0\}$ can be identified with a pair $(r_x,\omega_x)\in\mathbb R_{>0}\times\Omega$. For any $f\in\Lambda_k$ we have $f(x)=f(r_x,\omega_x)=r_x^kf^\circ(\omega_x)\in C(\mathbb R_{> 0}\times\Omega)$. Hence
\begin{multline*}
\int_{\tOmega}K_n(x,y)f(x)\tM(dx)-\frac{p_1^{n-k}f(y)}{(n-k)!}\\=r_y^n\int_{\tOmega}K^\circ_n(\omega_x,\omega_y)f^\circ(\omega_x)r_x^{n+k}(\M\otimes\gamma_{\frac{zz'}{\theta}})(dx)-\frac{r_y^nf^\circ(\omega_y)}{(n-k)!}\\=r_y^n\left((zz'\theta^{-1})_{n+k}\int_{\Omega}K_n^\circ(\omega,\cdot)f^\circ(\omega)\M(d\omega)-\frac{f^\circ(\cdot)}{(n-k)!}\right)\in r_y^n\Lambda_{\leq k-1}^\circ.
\end{multline*}
The required identity follows from fixing $r_y=1$.
\end{proof}

\begin{lem}
\label{ccc}
 Let $f_n$ and $g_n$ be two sequences from a vector space over $\mathbb C$. Assume that for some $c\in \mathbb C$ the following holds for any $n\geq 0$
\be
f_n=\sum_{m\geq 0}\frac{g_m}{(c)_{n+m}(n-m)!}.
\ee
Then
\begin{equation}
\label{aab}
g_m=\sum_{n=0}^m(-1)^{m-n}\frac{(c+2m-1)(c)_{m+n-1}}{(m-n)!}f_n.
\end{equation}
\end{lem}
The proof is given in \cite[Lemma 3.3]{Et}, and it is similar to the inclusion-exclusion formula.

\begin{prop}
\label{propaaa}
For every $m>0$ the following holds
\begin{equation}
\label{aaa}
G_m=\sum_{n=0}^m(-1)^{m-n}\frac{(\frac{zz'}{\theta}+2m-1)(\frac{zz'}{\theta})_{m+n-1}}{(m-n)!}K_n^\circ.
\end{equation}
\end{prop}
\begin{proof}

Recall that $\Lambda^\circ_{\leq m}$ is spanned by the elements $g_\mu$ with $|\mu|\leq m$, hence if $|\lambda|>m$ then $g_\lambda$ is orthogonal to $\Lambda^\circ_{\leq m}$. Then by Lemma \ref{Lemma}
\be
\int_{\Omega\times\Omega} K^\circ_n(\omega,\sigma)g_\lambda(\omega)g_\mu(\sigma)\M(d\omega)\M(d\sigma)=\frac{\delta_{\lambda,\mu}}{(n-|\lambda|)!(\frac{zz'}{\theta})_{|\lambda|+n}}.
\ee 
Since $\{g_\lambda(\omega)g_\mu(\sigma)\}_{\lambda,\mu}$ where $\lambda$ and $\mu$ run over partitions with no part equal to $1$ is an orthonormal basis of $L^2(\Omega\times\Omega, \M\otimes\M)$ we have
\begin{equation*}
K^\circ_n=\frac{1}{n!(\frac{zz'}{\theta})_{n}}+\sum_{m\geq 2}\frac{G_m}{(\frac{zz'}{\theta})_{n+m}(n-m)!}.
\end{equation*}
Together with Lemma \ref{ccc} this implies \eqref{aaa}.
\end{proof}

Recall that we want to prove the convergence of
\be
p(t,\sigma,\omega)=\sum_{\lambda}e^{-\alpha_{|\lambda|}t}g_\lambda(\omega)g_\lambda(\sigma)=1+\sum_{m\geq2}e^{-\alpha_mt}G_m(\omega,\sigma).
\ee
We will use Proposition \ref{propaaa} to give a sufficiently strong upper bound on $G_m$.

\begin{prop}
\label{bound}
There exist $C>0$ and $d>0$ such that
\be
||G_m||\leq Cm^{dm}
\ee
where $||\cdot||$ is the sup norm on $C(\Omega\times\Omega)$.
\end{prop}
\begin{proof}
Proposition \ref{propaaa} expresses $G_m$ in terms of the kernels $K^\circ_n$ which in turn are defined via functions $P_\lambda^\circ$ , so first we will give an upper bound for $P^\circ_\lambda$. 
From \eqref{dimexp} we have
\be
1=(p_1^n)^\circ=\left(\sum_{|\lambda|=n}\dim_\theta(\lambda)P_\lambda(\omega;\theta)\right)^\circ=\sum_{|\lambda|=n}\dim_\theta(\lambda)P_\lambda^\circ(\omega;\theta).
\ee
As pointed out in Section \ref{thomaS} the functions $P^\circ(\omega;\theta)$ are nonnegative, hence
\be
P^\circ_\lambda(\omega;\theta)\leq\dim_\theta(\lambda)^{-1}.
\ee
It is known (see \cite[section 5]{OkOl}) that 
\be
\dim_\theta(\lambda)=\frac{|\lambda|!}{H_\theta(\lambda)},
\ee
where 
\be
H_\theta(\lambda):=\prod_{(i,j)\in\lambda}(\lambda_i-j+\theta(\lambda'_j-i)+1).
\ee
Note that
\be
\lambda_i-j+\theta(\lambda'_j-i)+1\leq |\lambda|+\theta|\lambda|=|\lambda|(1+\theta),
\ee
hence
\be
P^\circ_\lambda(\omega;\theta)\leq\dim_\theta(\lambda)^{-1}=\frac{H_\theta(\lambda)}{|\lambda|!}\leq\frac{|\lambda|^{|\lambda|}(1+\theta)^{|\lambda|}}{|\lambda|!}.
\ee

In order to estimate $K^\circ_n$ we need bounds for $b_\lambda$ and $(z)_{\lambda,\theta}$. Recall that $(z,z')$ are from principal or complementary series hence there exists $\delta_z>0$ such that $|z+k+\theta l|>\delta_z$ for any $k,l\in\mathbb Z$, so $|(z)_{\lambda,\theta}|>\delta_z^{|\lambda|}$. 

For $b_\lambda^{(\theta^{-1})}$ we use \cite[VI, 10.10]{macdonald}, which gives
\be
b_\lambda^{(\theta^{-1})}=\prod_{(i,j)\in\lambda}\frac{\lambda_i-j+\theta(\lambda'_j-i+1)}{\lambda_i-j+1+\theta(\lambda'_j-i)}\leq (\theta+1)^{|\lambda|}.
\ee
This implies
\be
K_n^\circ(\omega,\sigma)\leq \rho(n)\frac{n^{2n}(1+\theta)^{3n}}{(n!)^2\delta_z^{n}\delta_{z'}^{n}},
\ee
where $\rho(n)$ is the number of partitions of the number $n$. Since $\rho(n)$ equals the number of conjugacy classes in the symmetric group of order $n$, we have $\rho(n)\leq n!$. Hence there exists some $C>0$ such that
\be
K_n^\circ(\omega, \sigma)\leq Cn^{3n}.
\ee
Finally, for $G_m$ we have
\begin{multline*}
|G_m|\leq\sum_{n=0}^m\frac{|\frac{zz'}{\theta}+2m-1|(\frac{zz'}{\theta})_{m+n-1}}{(m-n)!}|K_n^\circ|\leq C\sum_{n=0}^m(2m+\frac{zz'}{\theta})^{2m}n^{3n}\\
\leq  mC(2m+\frac{zz'}{\theta})^{2m}m^{3m}\leq Dm^{6m}
\end{multline*}
for some constant $D>0$.
\end{proof}

\begin{theo}
\label{result}
1) The process $X_{z,z',\theta}$ has a continuous transition density with respect to $\M$, that is there is a continuous function $p(t,\sigma,\omega)$ on $\mathbb R_{>0}\times\Omega\times\Omega$ such that for any $f\in C(\Omega)$ and $t>0$ the following identity holds
\be
T(t)f(\sigma)=\int_{\Omega}p(t,\sigma,\omega)f(\omega)\M(d\omega).
\ee

2) The transition density $p(t,\sigma,\omega)$ is given by the following series converging in the supremum norm
\be
p(t,\sigma,\omega)=1+ \sum_{n=0}^{\infty}\sum_{|\lambda|=n}C_n(t)\frac{P^\circ_\lambda(\sigma)P^\circ_\lambda(\omega)}{(z)_{\lambda.\theta}(z')_{\lambda,\theta}},
\ee
where the coefficients $C_n(t)$ are defined by
\be
C_n(t)=\sum_{\mathclap{\substack{m\geq n\\m\geq2}}}e^{-\alpha_mt}(-1)^{m-n}\frac{(\frac{zz'}{\theta}+2m-1)(\frac{zz'}{\theta})_{m+n-1}}{(m-n)!}
\ee
with $\alpha_m:=m(m-1+zz'\theta^{-1})$.
\end{theo}
\begin{proof}
As noted above, it is enough to prove that for any $\ve>0$ the series
\be
p(t,\sigma,\omega)=1+\sum_{m\geq 2}e^{-t\alpha_m}G_m(\sigma,\omega)
\ee
uniformly absolutely converges for $t\in(\ve,\infty)$. From Proposition \ref{bound} we have
\be
||e^{-t\alpha_m}G_m||\leq Ce^{-\ve m(m-1)+dm\log(m)}.
\ee
Note that 
\be
\log \(\frac{\ve m}{2d}\)\leq \frac{\ve m}{2d}-1,
\ee
hence
\be
dm\log m\leq \frac{\ve}{2}m^2-dm+dm\log\(\frac{\ve}{2d}\).
\ee
So we have
\be
||e^{-t\alpha_m}G_m||\leq Ce^{-c_1 m^2-c_2m}
\ee
for $c_1>0$ and the desired convergence follows. 
\end{proof}
\end{section}

\begin{section}{Ergodic theorem for the process with Jack parameter}

As shown in {\cite[Theorem 9.10(iii)]{olshanski}} the diffusion with Jack's parameter is ergodic in the following sense
\be
\lim_{t\to\infty}||T(t)f-\int_\Omega f(\omega)\M(d\omega)||_{\mathrm{sup}}=0
\ee

Following \cite[Remark 3.6]{Et}, the expression for the transition density can be used to strengthen the ergodic theorem. 
\begin{cor}
There is a constant $K$ depending only on $z,z',\theta$ such that for any $\sigma\in\Omega$ we have
\be
||P(t,\sigma,\cdot)-\M||_{\mathrm{var}}\leq Ke^{-t\alpha_2},
\ee
where $P(t,\sigma,\cdot)$ is the transition function of the process $X_{z,z',\theta}$, $||\cdot||_{\mathrm{var}}$ is the total variation and $\alpha_m=m(m-1+zz'\theta^{-1})$.
\end{cor}
\begin{proof}
Recall that for measures $\mu,\nu$ such that $\nu$ has density $f$ with respect to $\mu$ the following identity holds
\be
||\mu-\nu||_{\mathrm{var}}=\int |1-f|d\mu.
\ee
In our case $P(t,\sigma,\cdot)$ has the density $p(t,\sigma,\cdot)$ with respect to $\M$, so
\be
||P(t,\sigma,\cdot)-\M||_{\mathrm{var}}=\int_{\Omega} |p(t,\sigma,\omega)-1|\M(d\omega).
\ee
Using the expression for $p(t,\sigma,\omega)$ we have for $t>0$
\be
|p(t,\sigma,\omega)-1|\leq e^{-t\alpha_2}\sum_{m\geq 2}|G_m|e^{-t(\alpha_m-\alpha_2)}.
\ee
Moreover, the sum in the right hand side converges for any $t>0$ and is decreasing in $t$. Then for any $t_0>0$ and $t\geq t_0$ we have
\be
||P(t,\sigma,\cdot)-\M||_{\mathrm{var}}\leq K_0e^{-t\alpha_2},
\ee
where 
\be
K_0=\sum_{m\geq 2}||G_m||_{\mathrm{sup}}e^{-t(\alpha_m-\alpha_2)}.
\ee
Finally, note that the variance distance between any two measures is less or equal to $2$, hence for $t\geq 0$ and $K=K_0+2e^{-t\alpha_2}$ we have
\be
||P(t,\sigma,\cdot)-\M||_{\mathrm{var}}\leq Ke^{-t\alpha_2}.
\ee
\end{proof}

\begin{remark}
By giving a bound of $\int_{\Omega}|G_2(\sigma,\omega)|\M(d\omega)$ we may further improve the rate of convergence. Note that $K_n^\circ(\sigma,\omega)>0$ for any $\sigma,\omega\in\Omega$ so 
\be
|G_2|=\left|\sum_{n=0}^2(-1)^{2-n}\frac{(\frac{zz'}{\theta}+3)(\frac{zz'}{\theta})_{n+1}}{(2-n)!}K_n^\circ\right|\leq \sum_{n=0}^2\frac{(\frac{zz'}{\theta}+3)(\frac{zz'}{\theta})_{n+1}}{(2-n)!}K_n^\circ.
\ee
Note that from the Lemma \ref{Lemma} we have
\be
\int_{\Omega}K_n^\circ(\sigma,\omega)\M(d\omega)=\frac{1}{n!(\frac{zz'}{\theta})_n}.
\ee
Hence
\begin{multline*}
\int_{\Omega}|G_2(\sigma,\omega)|\M(d\omega)\leq\sum_{n=0}^2\frac{(\frac{zz'}{\theta}+3)(\frac{zz'}{\theta})_{n+1}}{(2-n)!}\frac{1}{n!(\frac{zz'}{\theta})_n}\\
=\sum_{n=0}^2\frac{(\frac{zz'}{\theta}+3)(\frac{zz'}{\theta}+n)}{(2-n)!n!}=2\(\frac{zz'}{\theta}+3\)\(\frac{zz'}{\theta}+1\).
\end{multline*}
Then for some constant $C>0$ depending only on $z,z'$ and $\theta$ the following holds
\be
||P(t,\sigma,\cdot)-\M||_{\mathrm{var}}\leq2\(\frac{zz'}{\theta}+1\)\(\frac{zz'}{\theta}+3\)e^{-t\alpha_2}+Ce^{-t\alpha_3}.
\ee

\end{remark}

\end{section}

\section*{Appendix A: Match with \cite{BaFo}.}

The purpose of this section is to match the notation appearing in Section 4 of \cite{BaFo} with our notation. 

As noted before, our parameter $\theta$ is the inverse of the conventional parameter $\alpha$, parameter $a$ from \cite{BaFo} is equal to $c-1$ used in our work. Then the operator $\tilde H^{(L)}$ from \cite{BaFo} is equal to our operator $D^{(L)}_N$.

Jack's polynomials $C_\kappa^{(\alpha)}$ used in \cite{BaFo} have different normalisation from ours, defined by
\be
p_1^n=\sum_{|\kappa|=n}C_{\kappa}^{(\alpha)}.
\ee

Comparing with \eqref{dimtheta}, we have
\be
C_{\kappa}^{(\alpha)}(x)=\dim_{\theta}(\kappa)P_{\kappa}(x;\theta).
\ee

Our Laguerre polynomials $L^{c,N,\theta}_\kappa$ differ from polynomials $L^a_{\kappa}(x_1,\dots, x_N;\alpha)$ used in \cite{BaFo} by the following rescaling
\be
L^{a}_\kappa(x_1,\dots,x_N;\alpha)=\frac{1}{\dim_\theta(\kappa)(N\theta)_\kappa}L^{c,N,\theta}_\kappa.
\ee

Finally, to get Propositions \ref{Lagpolex}, \ref{ortho} from \cite[Propositions 4.3, 4.10]{BaFo}, we also use the binomial formula from \cite{OkOl} in the following form
\be
\binom{\kappa}{\sigma}_\theta\frac{C^{(\alpha)}_\kappa(x_1,\dots, x_N)}{C^{(\alpha)}_\kappa(1, \dots, 1)}=\frac{|\kappa|!}{(|\kappa|-|\sigma|!)}\frac{\dim_\theta(\sigma,\kappa)}{\dim_\theta(\kappa)(N\theta)_\sigma}Q_\sigma(x_1,\dots,x_N;\theta).
\ee

\section*{Appendix B: Degeneration to Petrov's diffusion}
Taking the limit regime
\be
\theta\to0,\quad zz'\to0,\quad\theta^{-1}zz'\to\tau,\quad z+z'\to-a,
\ee
we can formally degenerate the process with the Jack parameter to Petrov's generalisation of Ethier-Kurtz diffusion \cite{Pe1}, see \cite[Remark 9.12]{olshanski}. In this regime the Thoma simplex degenerates to Kingman's simplex $\overline{\nabla}_\infty$ formed by sequences $x_1\geq x_2\geq\dots\geq0$ such that
\be
\sum_{i\geq1} x_i\leq1.
\ee
The functions $p_k^\circ$ degenerate to the moment coordinates $q_k=\sum_ix_i^k$, the probability pregenerator is given by
\be
\sum_{i,j\geq1}(i+1)(j+1)(q_{i+j}-q_iq_j)\frac{\partial^2}{\partial q_i\partial q_j}+\sum_{i\geq1}(i+1)[(i-a)q_{i-1}-(i+\tau)q_i]\frac{\partial}{\partial q_i},
\ee 
and the invariant measure $\M$ degenerates to the Poisson-Dirichlet distribution $PD(\alpha,\tau)$. 

Recall that $m_i(\mu)$ is the number of rows of length $i$ in diagram $\mu$. The Laguerre functions degenerate to
\be
\mathfrak{L}_{\lambda}^{PD(\alpha,\tau)}=\sum_{\mu\subseteq\lambda}(-1)^{|\lambda|-|\mu|}\frac{\dim_0(\mu,\lambda)}{(|\lambda|-|\mu|)!}\left(\prod_{x\in\lambda/\mu}\mathfrak{q}_{\alpha\tau}(x)\right)r_\mu m_\mu(x;\theta).
\ee
Here $\dim_0(\mu,\lambda)$ is the weighted sum of increasing paths from $\mu$ to $\lambda$ in Kingman graph described in \cite[Section 4]{KOO} and the coefficient $r_\mu$ is defined by 
\be
r_\mu=\frac{\prod_im_i(\mu)}{\prod_i\mu_i},
\ee
and $\mathfrak{q}_{\alpha\tau}(x)$ is defined by
\be
\mathfrak{q}_{\alpha\tau}(i,j)=
\begin{cases}
(j-1)(j-1-\alpha), &\text{if $j>1$},\\
\tau+\alpha(i-1),&\text{if $j=1$}.
\end{cases}
\ee
Functions $\mathfrak{L}_{\lambda}^{PD(\alpha,\tau)}$ form an orthogonal system for lifted Poisson-Dirichlet distribution $\widetilde{PD}(\alpha, \tau)=PD(\alpha, \tau)\otimes\gamma_\tau$.

The proof of Lemma \ref{Lemma} also works under that specialisation, so our approach can be used to get a new proof of the eigenfunction expansion of the transition density in the case of Petrov's diffusion.

\textsc{National Research University Higher School of Economics, Moscow, Russia (until September 2018);}

\textsc{Massachusetts Institute of Technology, Cambridge, MA, United States (since September 2018);}

\textit{E-mail address}: \texttt{shortkih@gmail.com}

\end{document}